\documentclass[reqno,twoside]{amsart}

\usepackage{amsfonts}
\usepackage{amsmath,amssymb,amsthm,amsthm}
\usepackage{mathtools}
\usepackage{etoolbox}
\usepackage{color}
\usepackage[english]{babel}
\usepackage{hyperref}
\usepackage{microtype}
\usepackage{fullpage}

\usepackage{tikz}
\usepackage{pgfplots}
\usepackage{pgfplotstable}
\usepackage{subfig}
\usetikzlibrary{matrix,external,fit}

\DisableLigatures{encoding = *, family = * }

\newtheorem{theorem}{Theorem}[section]

\newtheorem{definition}[theorem]{Definition}

\newtheorem{proposition}[theorem]{Proposition}

\newtheorem{remark}[theorem]{Remark}
\numberwithin{equation}{section}

\def\RR{{\mathbb{R}}}

\newcommand{\norm}[2]{{\left\|#1\right\|}_{#2}}
\newcommand{\fl}[2]{(-\Delta)^#1#2}
\newcommand{\cns}{C_{N,s}}

\title{The Poisson equation from non-local to local}

\author{Umberto Biccari\textsuperscript{1}}  
\address{\textsuperscript{1,2}\,DeustoTech, University of Deusto, 48007 Bilbao, Basque Country, Spain.}
\address{\textsuperscript{1,2}\,Facultad de Ingenier\'ia, Universidad de Deusto, Avenida de las Universidades 24, 48007 Bilbao, Basque Country, Spain, +34 944139003 - 3282.}
\email{umberto.biccari@deusto.es, u.biccari@gmail.com, victor.santamaria@deusto.es}
\thanks{The work of Umberto Biccari was partially supported by the Advanced Grant DYCON (Dynamic Control) of the European Research Council Executive Agency and by the MTM2014-52347 Grant of the MINECO (Spain).} 

\author{V\'ictor Hern\'andez-Santamar\'ia\textsuperscript{2}}
\thanks{The work of V\'ictor Hern\'andez-Santamar\'ia was partially supported by the
	Advanced Grant DYCON (Dynamic Control) of the European Research Council Executive Agency.} 

\keywords{Fractional Laplacian, elliptic equations, weak solutions}

\subjclass[2010]{35B30,35R11,35S05}

\begin{document}

\bibliographystyle{acm}

\begin{abstract}
We analyze the limit behavior as $s\to 1^-$ of the solution to the fractional Poisson equation $\fl{s}{u_s}=f_s$, $x\in\Omega$ with homogeneous Dirichlet boundary conditions $u_s\equiv 0$, $x\in\Omega^c$. 
We show that $\lim_{s\to 1^-} u_s =u$, with $-\Delta u =f$, $x\in\Omega$ and $u=0$, $x\in\partial\Omega$. Our results are complemented by a discussion on the rate of convergence and on extensions to the parabolic setting.
\end{abstract}

\maketitle

\section{Introduction and main result}\label{intro}

Let $0<s<1$ and let $\Omega\subset\RR^N$ be a bounded and regular domain. Let us consider the following elliptic problem
\begin{equation}\label{PE}
	\begin{cases}
	\fl{s}{u} = f, &x\in\Omega\tag{$\mathcal P_s$}
	\\
	u\equiv 0, & x\in\Omega^c.
	\end{cases}
\end{equation}

In \eqref{PE}, with $\fl{s}{}$ we indicate the fractional Laplace operator, defined for any function $u$ regular enough as the following singular integral
\begin{align}\label{frac_lapl}
	\fl{s}{u}(x):=\cns\,P.V.\int_{\RR^N} \frac{u(x)-u(y)}{|x-y|^{N+2s}}\,dy,
\end{align}
where $\cns$ is a normalization constant given by
\begin{align}\label{cns}
	\cns:= \frac{s2^{2s}\Gamma\left(\frac{N+2s}{2}\right)}{\pi^{N/2}\Gamma(1-s)},
\end{align}
$\Gamma$ being the usual Gamma function. Moreover, we have to mention that, for having a completely rigorous definition of the fractional Laplace operator, it is necessary to introduce also the class of functions $u$ for which computing $\fl{s}{u}$ makes sense. We postpone this discussion to the next section.

Models involving the fractional Laplacian or other types of non-local operators have been widely used in the description of several complex phenomena for which the classical local approach turns up to be inappropriate or limited. Among others, we mention applications in turbulence (\cite{bakunin2008turbulence}), elasticity (\cite{dipierro2015dislocation}), image processing (\cite{gilboa2008nonlocal}), laser beams design (\cite{longhi2015fractional}), anomalous transport and diffusion (\cite{meerschaert2012fractional}), porous media flow (\cite{vazquez2012nonlinear}). Also, it is well known that the fractional Laplacian is the generator of s-stable processes, and it is often used in stochastic models with applications, for instance, in mathematical finance (\cite{levendorskii2004pricing}).  

One of the main differences between these non-local models and  classical Partial Differential Equations is  that the fulfillment of a non-local equation  at a point involves the values of the function far away from that point. 

The Poisson problem \eqref{PE} is one of the most classical models involving the Fractional Laplacian, and it has been extensively studied in the past. Nowadays, there are many contributions concerning, but not limited to, existence and regularity of solutions, both local and global (\cite{biccari2017local,cozzi2017interior,grubb2015fractional,
leonori2015basic,ros2016boundary,ros2014dirichlet,servadei2014weak}), unique continuation properties (\cite{fall2014unique}), Pohozaev identities (\cite{ros2014pohozaev}), spectral analysis (\cite{frank2016refined}) and numerics (\cite{acosta2017fractional}).  

In the present paper, we are interested in analyzing the behavior of the solutions to \eqref{PE} under the limit $s\to 1^-$. Indeed, it is well-known (see, e.g., \cite{dihitchhiker,stinga2010extension}) that, at least for regular enough functions, it holds 
\begin{itemize}
	\item[$\bullet$] $\lim_{s\to 0^+}\fl{s}{u} = u$.
	
	\item[$\bullet$] $\lim_{s\to 1^-}\fl{s}{u} = -\Delta u$.
\end{itemize}

In view of that, it is interesting to investigate whether, when $s\to 1^-$, a solution $u_s$ to \eqref{PE} converges to a solution to the classical Poisson equation 
\begin{align}\label{poisson}
	\begin{cases}
	-\Delta u = f, &x\in\Omega\tag{$\mathcal P$}
	\\
	u = 0, & x\in\partial\Omega.
	\end{cases}
\end{align}

In our opinion, this is a very natural issue which, to the best of our knowledge, has never been fully addressed in the literature in the setting of weak solutions with minimal assumptions. As we will see, the answer to this question is positive. 

Before introducing our main result, let us recall that we have the following definition of weak solutions. 

\begin{definition}\label{weak_sol_def}
	Let $f\in H^{-s}(\Omega)$. A function $u\in H_0^s(\Omega)$ is said to be a weak solution of the Dirichlet problem \eqref{PE} if 
	\begin{align}\label{weak-sol}
	\frac{C_{N,s}}{2}\int_{\RR^N}\int_{\RR^N}\frac{(u(x)-u(y))(v(x)-v(y))}{|x-y|^{N+2s}}\;dxdy = \int_\Omega  fv\,dx
	\end{align}
	holds for every $v\in\mathcal{D}(\Omega)$.
\end{definition}

The main result of our work will be the following.

\begin{theorem}\label{limit_thm}
	Let $\mathcal{F}_s=\{f_s\}_{0<s<1}\subset H^{-s}(\Omega)$ be a sequence satisfying the following assumptions:
	\begin{itemize}
		\item[$\textbf{H1}$] $\norm{f_s}{H^{-s}(\Omega)}\leq C$, for all $0<s<1$ and uniformly with respect to $s$;
		
		\item[$\textbf{H2}$] $f_s\rightharpoonup f$ weakly in $H^{-1}(\Omega)$ as $s\to 1^-$.
	\end{itemize}		
	For all $f_s\in\mathcal{F}_s$, let $u_s\in H^s_0(\Omega)$ be the unique weak solution to the Dirichlet problem \eqref{PE}, in the sense of Definition \ref{weak_sol_def}. Then, as $s\to 1^-$, $u_s\to u$ strongly in $H^{1-\delta}_0(\Omega)$ for all $0<\delta\leq 1$. Moreover, $u\in H^1_0(\Omega)$ and verifies
\begin{align*}
	\int_{\Omega} \nabla u\cdot\nabla v\,dx = \int_{\Omega}fv\,dx, \;\;\; \forall v\in\mathcal{D}(\Omega),
\end{align*}
i.e. it is the unique weak solution to \eqref{poisson}.
\end{theorem}

The proof of Theorem \ref{limit_thm} will be based on classical PDEs techniques. Moreover, the result will follow from the limit behavior as $s\to 1^-$ of the operator $\fl{s}{}$ (\cite{dihitchhiker,stinga2010extension}) and of the norm $\norm{\cdot}{H^s(\Omega)}$ (\cite{bourgain2001another}).

Furthermore, notice that Theorem \ref{limit_thm} requires the existence of a sequence $\mathcal{F}_s$ satisfying the assumptions $\textbf{H1}$ and $\textbf{H2}$. We point out that such sequence indeed exists, and that it is possible to construct it systematically. We will give a proof of this fact in Section \ref{prel}.

This paper will be organized as follows: Section \ref{prel} will be devoted to introduce some preliminary definitions and results that will be needed in our analysis. In Section \ref{weak_sec}, instead, we will present the proof of Theorem \ref{limit_thm}, concerning the limit behavior of the solutions to \eqref{PE}. Finally, in Section \ref{rem_sec}, we will present an additional result of convergence under weaker assumptions, a discussion on the rate of approximation and an extension to the the parabolic setting.

\section{Preliminaries}\label{prel}

In this section, we introduce some preliminary results that will be useful for the proof of our main theorem.

We start by giving a more rigorous definition of the fractional Laplace operator, as we have anticipated in Section \ref{intro}. Define 
\begin{align*}
	\mathcal L_s^{1}(\RR^N):=\left\{u:\RR^N\to\RR\;\mbox{ measurable},\; \int_{\RR^N}\frac{|u(x)|}{(1+|x|)^{N+2s}}\;dx<\infty\right\}.
\end{align*}
For $u\in \mathcal L_s^{1}(\RR^N)$ and $\varepsilon>0$ we set
\begin{align*}
	(-\Delta)_\varepsilon^s u(x):= C_{N,s}\int_{\{y\in\RR^N:\;|x-y|>\varepsilon\}}\frac{u(x)-u(y)}{|x-y|^{N+2s}}\;dy,\;\;x\in\RR^N.
\end{align*}

The fractional Laplace operator $\fl{s}{}$ is then defined by the following singular integral:
\begin{align}\label{fl_def}
	\fl{s}{u}(x)=\cns\,\mbox{P.V.}\int_{\RR^N}\frac{u(x)-u(y)}{|x-y|^{N+2s}}\;dy=\lim_{\varepsilon\downarrow 0}(-\Delta)_\varepsilon^s u(x),\;\;x\in\RR^N,
\end{align}
provided that the limit exists. 

We notice that if $0<s<1/2$ and $u$ is smooth, for example bounded and Lipschitz continuous on $\RR^N$, then the integral in \eqref{fl_def} is in fact not really singular near $x$ (see e.g. \cite[Remark 3.1]{dihitchhiker}). Moreover, $\mathcal L_s^{1}(\RR^N)$ is the right space for which $ v:=(-\Delta)_\varepsilon^s u$ exists for every $\varepsilon>0$, $v$ being also continuous at the continuity points of  $u$. 

It is by now well-known (see, e.g., \cite{dihitchhiker}) that the natural functional setting for problems involving the Fractional Laplacian is the one of the fractional Sobolev spaces. Since these spaces are not so familiar as the classical integral order ones, for the sake of completeness, we recall here their definition. 

Given $\Omega\subset\RR^N$ regular enough and $s\in(0,1)$, the fractional Sobolev space $H^s({\Omega})$ is defined as
\begin{align*}
	H^s(\Omega):= \left\{u\in L^2(\Omega)\,:\, \frac{|u(x)-u(y)|}{|x-y|^{\frac N2+s}}\in L^2(\Omega\times\Omega)\right\}.
\end{align*}

It is classical that this is a Hilbert space, endowed with the norm (derived from the scalar product)
\begin{align*}
	\norm{u}{H^s(\Omega)} := \left(\norm{u}{L^2(\Omega)}^2 + |u|_{H^s(\Omega)}^2\right)^{\frac 12},
\end{align*}
where the term 
\begin{align*}
	|u|_{H^s(\Omega)}:= \left(\int_\Omega\int_\Omega \frac{|u(x)-u(y)|^2}{|x-y|^{N+2s}}\,dxdy\right)^{\frac 12}
\end{align*}
is the so-called Gagliardo seminorm of $u$. We set 
\begin{align*}
	H_0^s(\Omega):= \overline{C_0^\infty(\Omega)}^{\,H^s(\Omega)}
\end{align*}
the closure of the continuous infinitely differentiable functions compactly supported in $\Omega$ with respect to the $H^s(\Omega)$-norm. The following facts are well-known.
\begin{itemize}
	\item[$\bullet$] For $0<s\leq\frac 12$, the identity $H_0^s(\Omega) = H^s(\Omega)$ holds. This is because, in this case, the $C_0^\infty(\Omega)$ functions are dense in $H^s(\Omega)$ (see, e.g., \cite[Theorem 11.1]{jllions1972non}).
	
	\item[$\bullet$] For $\frac 12<s<1$, we have $H_0^s(\Omega)=\left\{ u\in H^s(\RR^N)\,:\,u=0\textrm{ in } \Omega^c\right\}$ (\cite{fiscella2015density}).
\end{itemize}

Finally, in what follows we will indicate with $H^{-s}(\Omega)=\left(H^s(\Omega)\right)'$ the dual space of $H^s(\Omega)$ with respect tot the pivot space $L^2(\Omega)$.

A more exhaustive description of fractional Sobolev spaces and of their properties can be found in several classical references (see, e.g., \cite{adams2003sobolev,dihitchhiker,jllions1972non}).

Coming back to our problem, let us recall that the existence and uniqueness of weak solutions to \eqref{PE} is guaranteed by the following result (see, e.g., \cite[Proposition 1.2.23]{peradottolaplaciano}).

\begin{proposition}\label{prop-ex}
Let $\Omega\subset\RR^N$ be an arbitrary bounded open set and $0<s<1$. Then for every $f\in H^{-s}(\Omega)$, the Dirichlet problem \eqref{PE} has a unique weak solution $u\in H_0^s(\Omega)$. Moreover, there exists a constant $C>0$ such that
\begin{align}\label{est-sol}
	\norm{u}{H_0^s(\Omega)}\le C\norm{f}{H^{-s}(\Omega)}.
\end{align}
In addition, we can take $C=\sqrt{2/\cns}$. 
\end{proposition}

We remind that our main interest in the present work is the analysis of the behavior of the solutions of \eqref{PE} when $s\to 1^-$. The proof of Theorem \ref{limit_thm} is obtained employing classical techniques in functional analysis, as well as the following results.

\begin{proposition}[{\cite[Proposition 4.4]{dihitchhiker}}]\label{limit_prop}
For any $u\in C_0^\infty(\RR^N)$ the following statements hold:
\begin{itemize}
	\item[(i)] $\lim_{s\to 0^+}\fl{s}{u} = u$.
	\item[(ii)] $\lim_{s\to 1^-} \fl{s}{u} = -\Delta u$.
\end{itemize}
\end{proposition}

\begin{proposition}[{\cite[Corollary 7]{bourgain2001another}}]\label{brezis_prop}
For any $\varepsilon>0$, let $g_\varepsilon \in H^{1-\varepsilon}(\Omega)$. Assume that 
\begin{align*}
	\varepsilon\norm{g_\varepsilon}{H^{1-\varepsilon}(\Omega)}^2\leq C_0,
\end{align*}
where $C_0$ is a positive constant not depending on $\varepsilon$. Then, up to a subsequence, $\{g_\varepsilon\}_{\varepsilon>0}$ converges in $L^2(\Omega)$ (and, in fact, in $H^{1-\delta}(\Omega)$, for all $\delta > 0$) to some $g\in H^1(\Omega)$.
\end{proposition}

Finally, as we pointed out in Section \ref{intro}, our main result requires a sequence $\mathcal{F}_s$ satisfying the assumptions $\textbf{H1}$ and $\textbf{H2}$. The existence of such a sequence is guaranteed by the following.

\begin{proposition}\label{limit_f}
	For any $f\in H^{-1}(\Omega)$ there exists a sequence $\mathcal{F}_s=\{f_s\}_{0<s<1}\subset H^{-s}(\Omega)$ verifying the assumptions
	\begin{itemize}
		\item[$\textbf{H1}$] $\norm{f_s}{H^{-s}(\Omega)}\leq C$, for all $0<s<1$ and uniformly with respect to $s$.
		
		\item[$\widehat{\textbf{H2}}$] $f_s\to f$ strongly in $H^{-1}(\Omega)$ as $s\to 1^-$.
	\end{itemize}		
\end{proposition}

\begin{proof}
Recall that any $f\in H^{-1}(\Omega)$ can be written as $f=\textrm{div}(g)$ with $g\in L^2(\Omega)$. Furthermore, let us introduce a standard mollifier $\rho_\varepsilon$ defined as 
\begin{align*}
	\rho_\varepsilon(x):= \begin{cases}
	C\varepsilon^{-N}\exp\left(\frac{\varepsilon^2}{|x|^2-\varepsilon^2}\right), & \textrm{if } |x|<\varepsilon
	\\
	0, & \textrm{if } |x|\geq\varepsilon
	\end{cases}
\end{align*}
and set $g_\varepsilon:=g\star\rho_\varepsilon$. It is classical that:
\begin{itemize}
		\item[(i)] $g_\varepsilon$ is well defined, since $g\in L^2(\Omega)$, hence it is locally integrable.
		\item[(ii)] $g_\varepsilon\in C_0^\infty(\Omega_\varepsilon)$, with $\Omega_\varepsilon:=\left\{x\in\Omega\,:\, \textrm{dist}(x,\partial\Omega)>\varepsilon\right\}.$
		\item[(iii)] $\partial_{x_i}g_\varepsilon$ is bounded uniformly with respect to $\varepsilon$ for all $i=1,\ldots,N$.
		\item[(iv)] $\lim_{\varepsilon\to 0^+} g_\varepsilon = g$, strongly in $L^2(\Omega)$. 
\end{itemize}
	
Thus we can take $f_\varepsilon:=\textrm{div}(g_\varepsilon)$ and, from Property (iii) above, we immediately have that $\norm{f_\varepsilon}{H^{-1+\varepsilon}(\Omega)}$ is bounded uniformly with respect to $\varepsilon$. In addition, using Properties (ii) and (iv), it is straightforward that, for all $i=1,\ldots,N$, $\partial_{x_i} g_\varepsilon = \rho_\varepsilon\star g_{x_i} \to g_{x_i}$ as $\varepsilon\to 0^+$. Hence, 
\begin{align*}
	\lim_{\varepsilon\to 0^+}f_\varepsilon = \lim_{\varepsilon\to 0^+}\textrm{div}(g_\varepsilon) = \textrm{div}(g) = f,
\end{align*}
where the convergence is strong in $H^{-1}(\Omega)$. Therefore, by choosing $\varepsilon=1-s$, following the above argument we can construct a sequence $\{f_s\}_{0<s<1}\subset H^{-s}(\Omega)$ verifying $\textbf{H1}$ and $\widehat{\textbf{H2}}$.
\end{proof}

\begin{remark}
	Notice that $\widehat{\textbf{H2}}$ is a property of strong convergence in $H^{-1}(\Omega)$ which, clearly, implies the weak convergence in the same functional setting (property $\textbf{H2}$). Therefore, Proposition \ref{limit_f} provides a sequence $\mathcal{F}_s$ which is within the hypotheses of Theorem \ref{limit_thm}.
\end{remark}

\section{The elliptic case: proof of Theorem \ref{limit_thm}}\label{weak_sec}
In this Section, we give the proof of Theorem \ref{limit_thm} employing the definition of weak solution that we gave in Section \ref{prel}. 

\begin{proof}[Proof of Theorem \ref{limit_thm}]
First of all, since we are interested in the behavior for $s\to 1^-$, until the end of the proof we will assume $s> 1/2$.  
Moreover, from $\textbf{H2}$ and the definition of weak convergence we get 
\begin{align}\label{limit-rhs}
	\lim_{s\to 1^-}\int_{\Omega} f_sv\,dx = \int_{\Omega} fv\,dx, \;\;\; \forall v\in\mathcal{D}(\Omega).
\end{align}

For all $0<s<1$, let $u_s\in H_0^s(\Omega)$ be the solution to \eqref{PE} corresponding to the right-hand side $f_s$. According to Proposition \ref{prop-ex}, for $s$ sufficiently close to one we have the estimate
\begin{align}\label{norm_est}
	\sqrt{1-s}\norm{u_s}{H^s(\Omega)}\leq\mathcal{C}(s,N)\norm{f_s}{H^{-s}(\Omega)},
\end{align}
with 
\begin{align*}
	\mathcal{C}(s,N) := \sqrt{\frac{2-2s}{\cns}}
\end{align*}

Moreover, for all $N$ fixed, the constant $\mathcal{C}(s,N)$ is decreasing as a function of $s$ (see Figure \ref{figure}). This of course implies
\begin{align*}
	\mathcal{C}(s,N) < \mathcal{C}\left(\frac 12,N\right) = \sqrt{\frac{\pi}{\Gamma\left(\frac{N+1}{2}\right)}}.
\end{align*}

\begin{figure}[h]
\centering
\pgfplotstableread{./cns_21-09-2017_11h22.org}{\datp}
\begin{tikzpicture}
  \begin{axis}[xlabel=$s$, ylabel = $\mathcal{C}(s{,}N)$, ylabel style={rotate=0}, xmin=1/2,xmax=1,xtick={0.5,1}, ytick ={0,2}, legend cell align = {left}, legend style ={draw=none}, legend pos= outer north east]
                 
		\addplot [color=black, densely dotted, thick] table[x=0,y=1]{\datp};\addlegendentry{\;$N=2$}
		\addplot [color=black, densely dashed, thick] table[x=0,y=3]{\datp};\addlegendentry{\;$\phantom{N}=4$} 
		\addplot [color=black, dashed, semithick] table[x=0,y=5]{\datp};\addlegendentry{\;$\phantom{N}=6$}
		\addplot [color=black, dashdotted, semithick] table[x=0,y=7]{\datp};\addlegendentry{\;$\phantom{N}=8$}
		\addplot [color=black, semithick] table[x=0,y=9]{\datp};\addlegendentry{\;$\phantom{N}=10$}

\end{axis}
\end{tikzpicture}
\caption{Behavior of $\mathcal{C}(s,N)$ as a function of $s\in\left[\frac 12,1\right]$ for different fixed values of $N$.}\label{figure}
\end{figure}
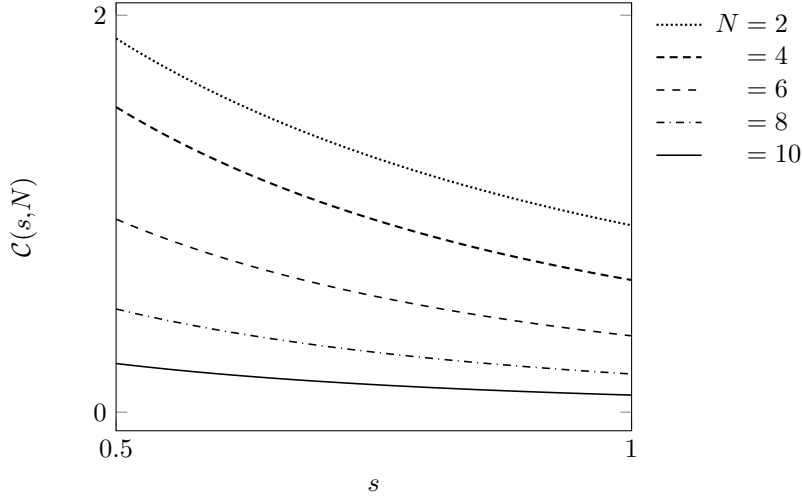

Therefore, from \eqref{norm_est} and the uniform boundedness of $\norm{f_s}{H^{-s}(\Omega)}$ we deduce that 
\begin{align*}
	\sqrt{1-s}\norm{u_s}{H^s(\Omega)}\leq C
\end{align*}
with $C$ depending only on $N$ and $\Omega$. This, thanks to Proposition \ref{brezis_prop}, allows us to conclude that $u_s\to u$ strongly in $H^{1-\delta}_0(\Omega)$ for any $0<\delta\leq 1$, and that $u\in H_0^1(\Omega)$.

Notice that, according to \cite[Section 6]{warma2015fractional}, for all $\phi\in H_0^s(\Omega)$ and $\psi\in\mathcal{D}(\Omega)$ we have the following identity
\begin{align*}
	\big\langle \fl{s}{\phi},\psi\big\rangle_{L^2(\Omega)} &=\frac{C_{N,s}}{2}\int_{\RR^N}\int_{\RR^N}\frac{(\phi(x)-\phi(y))(\psi(x)-\psi(y))}{|x-y|^{N+2s}}\;dxdy 
	\\ 
	&= \big\langle \phi,\fl{s}{\psi}\big\rangle_{L^2(\Omega)}.
\end{align*}

This can be applied to the variational formulation \eqref{weak-sol}, which can thus be rewritten as
\begin{align}\label{weak-def-s}
	\big\langle u_s,\fl{s}{v}\big\rangle_{L^2(\Omega)} = \int_\Omega f_sv\,dx.
\end{align}
Moreover, since $v\in\mathcal{D}(\Omega)$ we have
\begin{align*}
	\big|u_s\fl{s}{v}\big| \leq C|u_s|,
\end{align*}
where, clearly, $u_s\in L^2(\Omega)\hookrightarrow L^1(\Omega)$, being $\Omega$ a bounded domain. Hence we can use the Dominated Convergence Theorem and Proposition \ref{limit_prop} to conclude that
\begin{align*}
	\lim_{s\to 1^-} \big\langle u_s,\fl{s}{v}\big\rangle_{L^2(\Omega)} = \lim_{s\to 1^-} \int_\Omega u_s\fl{s}{v}\,dx = -\int_\Omega u\Delta v\,dx = \int_\Omega \nabla u\cdot\nabla v\,dx.
\end{align*}
This, together with \eqref{limit-rhs} and \eqref{weak-def-s} implies that $u$ verifies
\begin{align*}
	\int_\Omega \nabla u\cdot\nabla v\,dx = \int_{\Omega} fv\,dx, \;\;\; \forall v\in\mathcal{D}(\Omega),
\end{align*}
i.e. it is a weak solution to \eqref{poisson}.
\end{proof}

\begin{remark}
The result that we just proved is to some extent not surprising, due to the limit behavior of the fractional Laplacian as $s\to 1^-$. In fact, a hint that Theorem \ref{limit_thm} had to be true comes from the very classical example
\begin{equation*}
	\begin{cases}
		\fl{s}{u_s} =1, &x\in B(0,1)
		\\
		u_s\equiv 0, & x\in B(0,1)^c,
	\end{cases}
\end{equation*}
whose solution is given explicitly by
\begin{align*}
	u_s(x) = \frac{2^{-2s}\Gamma\left(\frac N2\right)}{\Gamma\left(\frac{N+2s}{2}\right)\Gamma(1+s)}\left(1-|x|^2\right)^s\chi_{B(0,1)}.
\end{align*}
Indeed, it can be readily checked that, for $x\in B(0,1)$, 
\begin{align*}
	\lim_{s\to 1^-} u_s(x) = \frac{1}{2N}\left(1-|x|^2\right):=u(x),
\end{align*}
which is the unique solution to the limit problem 
\begin{align*}
	\begin{cases}
		-\Delta u =1 , &x\in B(0,1)
		\\
		u= 0, & x\in \partial B(0,1).
	\end{cases}
\end{align*}

Of course, the above fact does not tell anything about the general case of problem \eqref{PE}. To the best of our knowledge, this is an issue that, although natural and probably expected, has not yet been fully addressed in the literature (at least, not in the setting of weak solutions with minimal assumptions) and our contribution helps to fill in this gap. 
\end{remark}

\section{Additional results an further comments}\label{rem_sec}

\subsection{Weakening the assumptions of Theorem \ref{limit_thm}}

Scope of this section is to show that a convergence result in the spirit of Theorem \ref{limit_thm} can be obtained under weaker assumption on the sequence $\mathcal F_s$ of the right-hand sides of \eqref{PE}. In particular, we are going to prove the following.

\begin{theorem}\label{limit_thm_weak}
Let $\mathcal F_s=\{f_s\}_{0<s<1}\subset H^{-1}(\Omega)$ be a sequence such that $f_s\rightharpoonup f$ weakly in $H^{-1}(\Omega)$. For all $f_s\in\mathcal F_s$, let $u_s$ be the corresponding solution to \eqref{PE}. Then, as $s\to 1^-$, $u_s\rightharpoonup u$ weakly in $L^2(\Omega)$, with $u$ solution to \eqref{poisson} in the transposition sense.
\end{theorem}

\begin{proof}
First of all, since we are interesting in analyzing the behavior of $u_s$ as $s\to 1^-$, until the end of this proof we will always assume $s>1/2$. Moreover, observe that, the right-hand side $f_s$ belongs to $H^{-1}(\Omega)$, which is strictly greater than $H^{-s}(\Omega)$. Therefore, we cannot apply Lax-Milgram Theorem. Instead, we shall define the solution to \eqref{PE} in a different way.

For all $\phi\in L^2(\Omega)$, let $y$ be solution of the elliptic problem 
\begin{align}\label{PE_transp}
	\begin{cases}
		\fl{s}{y}=\phi, & x\in\Omega
		\\
		y \equiv 0, & x\in\Omega^c.
	\end{cases}
\end{align}

Recall that, due to the regularity of $\phi$ and to the results contained in \cite{biccari2017local,cozzi2017interior}, for all $\varepsilon>0$ we have $y\in H^{2s-\varepsilon}_0(\Omega)\hookrightarrow H^1_0(\Omega)$, with continuous and compact embedding. 

Moreover, the map $\Lambda: \phi\mapsto y$ is linear and continuous from $L^2(\Omega)$ into $H^{2s-\varepsilon}_0(\Omega)$. Thus, $\Lambda$ is compact from $L^2(\Omega)$ into $H^1_0(\Omega)$ and its adjoint $\Lambda^*$ is a compact operator from $H^{-1}(\Omega)$ into $L^2(\Omega)$. In addition, 
\begin{align*}
	\langle f_s,y\rangle_{H^{-1}(\Omega),H^1_0(\Omega)} = \langle f_s,\Lambda \phi\rangle_{H^{-1}(\Omega),H^1_0(\Omega)} = (\Lambda^* f_s,\phi)_{L^2(\Omega)}.
\end{align*}

Therefore, $u_s:=\Lambda^*f_s\in L^2(\Omega)$ is a solution defined by transposition to \eqref{PE}, i.e. it satisfies
\begin{align}\label{transp_def_s}
	\int_\Omega u_s\phi\,dx = \langle f_s,y\rangle_{H^{-1}(\Omega),H^1_0(\Omega)}. 
\end{align}
Moreover, we have 
\begin{align}\label{us_norm}
	\norm{u_s}{L^2(\Omega)} \leq C\norm{f_s}{H^{-1}(\Omega)}.
\end{align}

In particular, $\{u_s\}_{0<s<1}$ is a bounded sequence in $L^2(\Omega)$, which implies that $u_s\rightharpoonup u$ weakly in $L^2(\Omega)$.

Notice that \eqref{transp_def_s} is obtained multiplying \eqref{PE} for $y$ and integrating over $\Omega$. Observe also that in this expression the functional spaces involved (namely $L^2(\Omega)$, $H_0^1(\Omega)$ and $H^{-1}(\Omega)$) do not depend on $s$. This, joint with \eqref{us_norm} and with the fact that $\phi\in L^2(\Omega)$, $f_s\in H^{-1}(\Omega)$ and $y\in H_0^1(\Omega)$, allows us to take the limit as $s\to 1^-$ in \eqref{transp_def_s}. Thanks to the Dominated Convergence Theorem, we then recover the expression
\begin{align}\label{transp_def}
	\int_\Omega u\phi\,dx = \langle f,y\rangle_{H^{-1}(\Omega),H^1_0(\Omega)},
\end{align}
i.e. $u$ is a solution by transposition to \eqref{poisson}. Moreover, since the $L^2(\Omega)$-regularity of $u_s$ cannot be improved, its convergence to a solution to \eqref{poisson} can be expected only in the weak sense.
\end{proof}

\subsection{Remarks on the convergence rate}

Our interest in the subject of this paper is motivated by previous results concerning the numerical approximation of the fractional Laplacian. In more detail, the issue that we addressed came from the observation that for the stiffness matrix $\mathcal A_h^s$ derived in \cite{biccari2017controllability} from the FE discretization of \eqref{frac_lapl} in dimension $N=1$  the following holds:
\begin{itemize}
	\item[(i)] $\lim_{s\to 0^+}\mathcal A_h^s = h\textrm{Tridiag}(1/6,2/3,1/6):=\mathcal I_h$, an approximation of the identity;
	
	\item[(ii)] $\lim_{s\to 1^-}\mathcal A_h^s = h^{-1}\textrm{Tridiag}(-1,2,-1):=\mathcal A_h$, the classical tridiagonal matrix for the FE approximation of the one-dimensional Laplacian.
\end{itemize}

The second property in particular implies that also the numerical solution $u_h^s$ associated to $\mathcal A_h^s$ converges to the one corresponding to $\mathcal A_h$. Therefore, investigating whether this still holds in the continuous case was a question that arose naturally. 

While we answered to this question in Theorem \ref{limit_thm}, there we did not specify under which rate this convergence occurs. In what follow, we present an informal discussion on this particular point. 

During the proof of Theorem \eqref{limit_thm}, we showed that the sequence $\{u_s\}_{0<s<1}$ of solutions to \eqref{PE} is bounded in $H_0^s(\Omega)$, with the following estimate 
\begin{align}\label{us_norm_est}
	\sqrt{1-s}\norm{u_s}{H^s(\Omega)}\leq C,
\end{align}
with $C$ a constant uniform with respect to $s$. This last inequality, in turn, was obtained as a consequence of Proposition \ref{prop-ex} and of the assumption $\mathbf{H1}$ on the sequence $\{f_s\}_{0<s<1}$ of the right-hand sides. 

Moreover, the factor $\sqrt{1-s}$ in \eqref{us_norm_est} already appears in \cite{bourgain2001another} to correct the well-known defect
of the seminorm $|\cdot|_{H^s(\Omega)}$ which, as $s\to 1^-$, does not converge to $|\cdot|_{H^1(\Omega)}$. 

In fact, if $\zeta$ is any smooth non-constant function, then for all $1< p<\infty$ we have $|\zeta|_{W^{s,p}(\Omega)}\to +\infty$ as $s\to 1^-$. This situation may be rectified by multiplying by $(1 -s)^{1/p}$ in front of $|\zeta|_{W^{s,p}(\Omega)}\to +\infty$. IN particular, we have 
\begin{align*}
	\lim_{s\to 1^-} (1-s)^{\frac 1p}|\zeta|_{W^{s,p}(\Omega)} = \left(\int_\Omega |\nabla\zeta|^p\,dx\right)^{\frac 1p}. 
\end{align*}

In view of these observations, we claim that the convergence that we obtained in Theorem \ref{limit_thm} satisfies the rate
\begin{align*}
	\lim_{s\to 1^-}\norm{u_s-u}{H^s(\Omega)} \sim \mathcal O(\sqrt{1-s}).
\end{align*}

Indeed, if this convergence were slower, then we would still have blow-up phenomena in the $H^s(\Omega)$-seminorm. On the other hand, if the convergence were faster, then for some $\alpha>1/2$
\begin{align*}
	\lim_{s\to 1^-}(1-s)^{\alpha}|\cdot|_{H^s(\Omega)} = \lim_{s\to 1^-}\underbrace{(1-s)^{\alpha-\frac 12}}_{\to 0}\underbrace{\sqrt{1-s}\,|\cdot|_{H^s(\Omega)}}_{\to |\cdot|_{H^1(\Omega)}}  = 0.
\end{align*}

Clearly, the discussion that we just presented is not a rigorous proof of our claim. Nevertheless, we believe that our statement is true, and a further confirmation is given by the following numerical simulations, where we compared the solution to \eqref{PE} and \eqref{poisson} for different values of $s$ and we computed the approximation error in the $H^s(\Omega)$-norm. As expected, we observe a convergence of $u_s$ to $u$, with a rate of $\sqrt{1-s}$.

\begin{figure}[!h]
	\pgfplotstableread{test1.org}{\datpu}
	
	\subfloat[Solutions to $\fl{s}{u_s}=\sin(\pi x^2)$ for different values of $s\in{[1/2,1]}$.]{
		\begin{tikzpicture}[scale=0.75]
		\begin{axis}[xmin=-0.98, xmax=0.98, ymax=0.7, xlabel=$x$, xtick={-0.98,0,0.98}, xticklabels={-1,0,1}, ytick=\empty, legend cell align = {left}, legend style ={draw=none}, legend pos= outer north east]
		
		\addplot [color=black, dotted, thick] table[x=0,y=3]{\datpu};\addlegendentry{\;$s=0.5$}
		\addplot [color=black, densely dotted, thick] table[x=0,y=4]{\datpu};\addlegendentry{\;$s=0.6$} 
		\addplot [color=black, densely dashed, semithick] table[x=0,y=5]{\datpu};\addlegendentry{\;$s=0.75$}
		\addplot [color=black, dashed, semithick] table[x=0,y=6]{\datpu};\addlegendentry{\;$s=0.9$}
		\addplot [color=black, dashdotted, semithick] table[x=0,y=7]{\datpu};\addlegendentry{\;$s=0.95$}
		\addplot [color=black, semithick] table[x=0,y=8]{\datpu};\addlegendentry{\;$s=1$}
		
		\end{axis}
		\end{tikzpicture}
	}
	\hspace{0.3cm}
	\subfloat[Decay of $\norm{u_s-u}{H^s(-1,1)}$ with respect to $s\in{[1/2,1]}$.]{
		\begin{tikzpicture}[scale=0.75]
		\begin{loglogaxis}[xmin=0.5, xmax=0.995, ymax=0.7, xlabel=$s$, xtick={0.5,0.995}, xticklabels={0.5,1}, legend cell align = {left}, legend style ={draw=none}, legend pos= outer north east]
		
		\addplot [color=black, semithick] table[x=1,y=2]{\datpu};\addlegendentry{\;error}
		\addplot [domain=0.5:1, dashed, samples=100, color=black, semithick]{0.05*sqrt(1-x)};\addlegendentry{\;$\sqrt{1-s}$}
		
		\end{loglogaxis}
		\end{tikzpicture}
	}
	\caption{Convergence of the solutions to $\fl{s}{u_s}=\sin(\pi x^2)$ with Dirichlet homogeneous boundary conditions as $s\to 1^-$, and its corresponding error in the $H^s(-1,1)$-norm.}
	\label{behavior1_fig}
\end{figure}
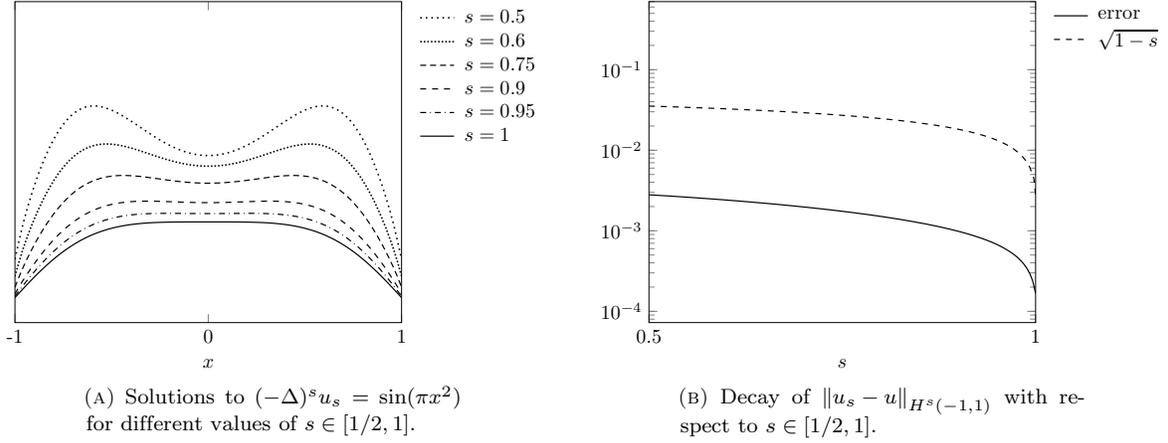

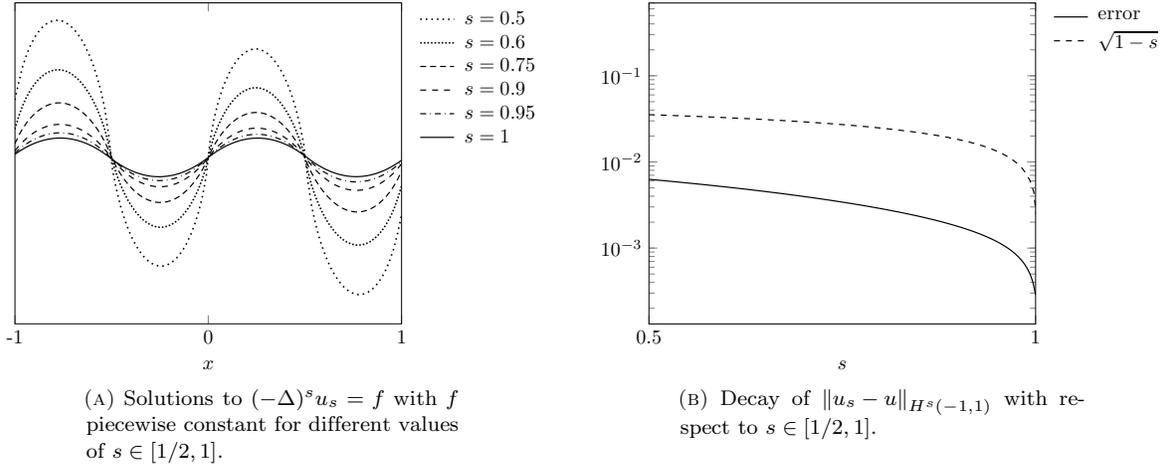
\begin{figure}[!h]
	\pgfplotstableread{test2.org}{\datpu}
	
	\subfloat[Solutions to $\fl{s}{u_s}=f$ with $f$ piecewise constant for different values of $s\in{[1/2,1]}$.]{
		\begin{tikzpicture}[scale=0.75]
		\begin{axis}[xmin=-0.98, xmax=0.98, ymax=0.25, xlabel=$x$, xtick={-0.98,0,0.98}, xticklabels={-1,0,1}, ytick=\empty, legend cell align = {left}, legend style ={draw=none}, legend pos= outer north east]
		
		\addplot [color=black, dotted, thick] table[x=0,y=3]{\datpu};\addlegendentry{\;$s=0.5$}
		\addplot [color=black, densely dotted, thick] table[x=0,y=4]{\datpu};\addlegendentry{\;$s=0.6$} 
		\addplot [color=black, densely dashed, semithick] table[x=0,y=5]{\datpu};\addlegendentry{\;$s=0.75$}
		\addplot [color=black, dashed, semithick] table[x=0,y=6]{\datpu};\addlegendentry{\;$s=0.9$}
		\addplot [color=black, dashdotted, semithick] table[x=0,y=7]{\datpu};\addlegendentry{\;$s=0.95$}
		\addplot [color=black, semithick] table[x=0,y=8]{\datpu};\addlegendentry{\;$s=1$}
		
		\end{axis}
		\end{tikzpicture}
	}
	\hspace{0.3cm}
	\subfloat[Decay of $\norm{u_s-u}{H^s(-1,1)}$ with respect to $s\in{[1/2,1]}$.]{
		\begin{tikzpicture}[scale=0.75]
		\begin{loglogaxis}[xmin=0.5, xmax=0.995, ymax=0.7, xlabel=$s$, xtick={0.5,0.995}, xticklabels={0.5,1}, legend cell align = {left}, legend style ={draw=none}, legend pos= outer north east]
		
		\addplot [color=black, semithick] table[x=1,y=2]{\datpu};\addlegendentry{\;error}
		\addplot [domain=0.5:1, dashed, samples=100, color=black, semithick]{0.05*sqrt(1-x)};\addlegendentry{\;$\sqrt{1-s}$}
		
		\end{loglogaxis}
		\end{tikzpicture}
	}
	\caption{Convergence of the solutions to $\fl{s}{u_s}=f$ with $f$ piecewise constant and Dirichlet homogeneous boundary conditions as $s\to 1^-$, and its corresponding error in the $H^s(-1,1)$-norm.}
	\label{behavior2_fig}
\end{figure}

\subsection{The parabolic case}

As it most often happens, the properties of the solutions to elliptic problems can be naturally transferred into the parabolic setting. In our case, this translates in the fact that the solution $\phi_s$ to the fractional heat equation

\begin{align}\label{FE}
	\begin{cases}
		\partial_t\phi_s + \fl{s}{\phi_s} = g_s, &(x,t)\in\Omega\times(0,T)\tag{$\mathcal H_s$}
		\\
		\phi_s\equiv 0, & (x,t)\in\Omega^c\times(0,T)
		\\
		\phi_s(x,0) = 0, & x\in\Omega,
	\end{cases}
\end{align}
converges as $s\to 1^-$ to the one to the local problem 
\begin{align}\label{HE}
	\begin{cases}
		\partial_t\phi -\Delta\phi = g, &(x,t)\in\Omega\times(0,T)\tag{$\mathcal H$}
		\\
		\phi= 0, & (x,t)\in\partial\Omega\times(0,T)
		\\
		\phi(x,0) = 0, & x\in\Omega.
	\end{cases}
\end{align}

First of all, let us recall that we have the following definition of weak solution for the parabolic problem \eqref{FE} (see, e.g., \cite{leonori2015basic}).
\begin{definition}\label{weak_sol_def_parabolic}
Let $g_s\in L^2(0,T;H^{-s}(\Omega))$. A function $\phi_s\in L^2(0,T;H_0^s(\Omega))\cap C([0,T];L^2(\Omega))$ with $\partial_t\phi_s\in L^2(0,T;H^{-s}(\Omega))$ is said to be a weak solution to the parabolic problem \eqref{FE} if for every $\psi\in\mathcal{D}(\Omega\times(0,T))$, it holds the equality
\begin{align}\label{weak-sol-par}
	\int_0^T \int_\Omega\partial_t\phi_s\psi\,dxdt &+ \frac{C_{N,s}}{2}\int_0^T\int_{\RR^N}\int_{\RR^N}\frac{(\phi_s(x)-\phi_s(y))(\psi(x)-\psi(y))}{|x-y|^{N+2s}}\;dxdydt \notag
	\\
	&= \int_0^T\int_\Omega g_s\psi\,dxdt.
\end{align}
\end{definition}

Moreover, thanks to \cite[Theorem 26]{leonori2015basic}, existence and uniqueness of solutions is guaranteed. Namely, we have

\begin{proposition}
Assume that $f_s\in L^2(0,T;H^{-s}(\Omega))$. Then problem \eqref{FE} has a unique finite energy solution, defined according to \eqref{weak_sol_def_parabolic}.
\end{proposition}

Then, adapting the methodology for the proof of Theorem \eqref{limit_thm}, the following result is immediate.

\begin{theorem}\label{limit_thm_parabolic}
	Let $\mathcal{G}_s:=\{g_s\}_{0<s<1}\subset L^2(0,T;H^{-s}(\Omega))$ be a sequence satisfying the following assumptions for all $0<t<T$:
	\begin{itemize}
		\item[$\textbf{K1}$] $\norm{g_s(t)}{H^{-s}(\Omega)}\leq C$, for all $0<s<1$ and uniformly with respect to $s$.
		
		\item[$\textbf{K2}$] $g_s(t)\rightharpoonup g(t)$ weakly in $H^{-1}(\Omega)$ as $s\to 1^-$.
	\end{itemize}		
For any $f_s\in\mathcal{G}_s$, let $\phi_s\in L^2(0,T;H^s_0(\Omega))$ be the unique weak solution to the corresponding parabolic problem \eqref{FE} in the sense of Definition \ref{weak_sol_def_parabolic}. Then, as $s\to 1^-$, $(\phi_s,\partial_t\phi_s)\to(\phi,\partial_t\phi)$ strongly in $L^2(0,T;H^{1-\delta}_0(\Omega))\times L^2(0,T;H^{-1}(\Omega))$ for any $0<\delta\leq 1$. Moreover, $\phi\in L^2(0,T;H^1_0(\Omega))\times L^2(0,T;H^{-1}(\Omega))$ and verifies
\begin{align*}
	\int_0^T \int_\Omega\partial_t\phi\psi\,dxdt + \int_0^T\int_\Omega \nabla\phi\cdot\nabla\psi\;dxdt = \int_0^T\int_\Omega g\psi\,dxdt, \;\;\; \forall\psi\in\mathcal{D}(\Omega\times(0,T)),
\end{align*}
i.e. it is the unique weak solution to \eqref{HE}.
\end{theorem}

\begin{proof}
First of all, notice that a sequence $\mathcal{G}_s$ verifying $\textbf{K1}$ and $\textbf{K2}$ exists. In fact, it can be constructed following the methodology of Proposition \ref{limit_f}, since both properties are independent of the time variable. Moreover, it is evident that we shall only analyze  the first term on the left-hand side of \eqref{weak-sol-par}. This is due to the following two facts:
\begin{itemize}
	\item The functional space in which the integration in time is carried out is fixed and does not depend on $s$. Therefore, the limit process does not affect the regularity in the time variable.
	
	\item For the remaining two terms in \eqref{weak-sol-par}, the limit as $s\to 1^-$ can be addressed in an analogous way as in the proof of Theorem \ref{limit_thm}.
\end{itemize}

On the other hand, since $\partial_t\phi_s\in L^2(0,T;H^{-s}(\Omega))$, the same argument previously developed for dealing with the term 
\begin{equation*}
	\int_0^T\int_\Omega g_s\psi\,dxdt
\end{equation*}
applies also to
\begin{equation*}
	\int_0^T\int_{\Omega} \partial_t\phi_s\psi\,dxdt.
\end{equation*}

In this way, we immediately conclude that, as $s\to 1^-$, $(\phi_s,\partial_t\phi_s)\to (\phi,\partial_t\phi)$ strongly in $L^2(0,T;H^{1-\delta}_0(\Omega))\times L^2(0,T;H^{-1}(\Omega))$ for all $0<\delta\leq 1$ and, in particular, that 
\begin{align*}
	\lim_{s\to 1^-}\int_0^T\int_{\Omega} \partial_t\phi_s\psi\,dxdt = \int_0^T\int_{\Omega} \partial_t\phi\psi\,dxdt.
\end{align*}
This, together with the above remarks, implies that the function $\phi$ satisfies
\begin{align*}
	\int_0^T \int_\Omega\partial_t\phi\psi\,dxdt + \int_0^T\int_\Omega \nabla\phi\cdot\nabla\psi\;dxdt = \int_0^T\int_\Omega g\psi\,dxdt, \;\;\; \forall\psi\in\mathcal{D}(\Omega\times(0,T)),
\end{align*}
i.e. it is the unique weak solution to \eqref{HE}.
\end{proof}

\section*{Acknowledgments}

The authors wish to acknowledge Enrique Zuazua (Universidad Auto\'onoma de Madrid, DeustoTech and Laboratoire Jacques-Louis Lions) for having suggested the topic of this work. Moreover, a special thank goes to Xavier Ros-Oton (Universit\"at Z\"urich) and Enrico Valdinoci (University of Melbourne) for interesting and clarifying discussions.
\bibliography{biblio}

\end{document}